\documentclass[review]{amsart}

\usepackage{hyperref}

\usepackage{color}
\usepackage{amsthm}
\usepackage{amsmath}
\usepackage{amsfonts}
\usepackage{amstext}
\usepackage[latin1]{inputenc}
\usepackage{amscd}
\usepackage{latexsym}
\usepackage{bm}
\usepackage{amssymb}
\usepackage[all]{xy}
\usepackage{euscript}
\usepackage{a4wide}
\usepackage{amsmath,amssymb,graphicx}
\usepackage{amssymb}
\usepackage{amsmath,boites}
\usepackage{caption}
\usepackage{subcaption}
\usepackage{graphics} 

\usepackage{graphicx}
\usepackage{boites}

\newcommand{\N}{\mathbb{N}}
\newcommand{\R}{\mathbb{R}}

\newcommand{\erre}{\mathbb{R}}

\newcommand\uno{{(1)}}
\newcommand\due{{(2)}}

\newcommand\kkk{{(k)}}

\newtheorem{example}{Example}
\newtheorem{definition}{Definition}
\newtheorem{proposition}{Proposition}
\newtheorem{theorem}{Theorem}
\newtheorem{lemma}{Lemma}

\newtheorem{remark}{Remark}

\begin{document}

\title{Hardy Type Inequalities for $\Delta_\lambda$-Laplacians}
\author{Alessia E. Kogoj}
\email{alessia.kogoj@unibo.it}
\address{Dipartimento di Matematica, Universit\`a di Bologna\\
Piazza di Porta San Donato, 5, IT-40126 Bologna  - Italy}

\author{Stefanie Sonner}
\email{sonner@mathematik.uni-kl.de}
\address{Felix-Klein-Center for Mathematics, University of Kaiserslautern\\
Paul-Ehrlich-Str. 31, D-67663 Kaiserslautern - Germany\\
BCAM - Basque Center for Applied Mathematics\\
Mazarredo, 14 E-48009 Bilbao, Basque Country - Spain}
\footnote{Until May 2014 the second author was supported by the ERC Advanced Grant FPT-246775 NUMERIWAVES.}


\maketitle

\begin{abstract}
We derive Hardy type inequalities for a large class of
sub-elliptic operators that belong to the class of $\Delta_\lambda$-Laplacians and find explicit values for the constants involved.
Our results generalize previous  inequalities 
obtained for Grushin type operators 
$$
\Delta_{x}+ |x|^{2\alpha}\Delta_{y},\qquad\ (x,y)\in\mathbb{R}^{N_1}\times\mathbb{R}^{N_2},\  \alpha\geq 0,
$$
which were proved to be sharp.
\end{abstract}

\section{Introduction}\label{sec_intro}

Let $\Omega\subset \R^N$ be a domain, where $N\geq 3.$ 
The $N$-dimensional version of the
classical Hardy inequality states that there exists a constant $c>0$ such that 
$$
c\int_\Omega\frac{|u(x)|^2}{|x|^2}dx\leq \int_\Omega |\triangledown u(x)|^2dx,
$$
for all  $u\in H_0^1(\Omega).$
If the origin $\{0\}$ belongs to the set $\Omega,$ 
the optimal constant is $c=\left(\frac{N-2}{2}\right)^2,$ 
but not attained in $H_0^1(\Omega).$ 
Hardy originally proved this inequality in 1920 for the one-dimensional case.

Hardy inequalities are an important tool in the analysis of linear and non-linear PDEs (see, e.g., \cite{BrVa},\cite{DAm2},\cite{VaZu}),
and over the years the classical Hardy inequality has been improved and extended in many directions.  
Our aim is to derive Hardy type inequalities for a class of degenerate elliptic operators extending 
previous results by D'Ambrosio in \cite{DAm}. 
He obtained a family of Hardy type inequalities for the Grushin type operator
$$
\Delta_{x}+ |x|^{2\alpha}\Delta_{y},\qquad \ (x,y)\in\R^{N_1}\times\R^{N_2},
$$
where $\alpha$ is a real positive constant. 
The class of operators we consider contains Grushin type operators and, e.g., operators of the form
\begin{eqnarray*}
 \Delta_{x}+ |x|^{2\alpha}\Delta_{y} + |x|^{2\beta}|y|^{2\gamma} \Delta_{z},
\qquad (x,y,z)\in\R^{N_1}\times\R^{N_2}\times\R^{N_3},
\end{eqnarray*}
 where $\alpha,\beta$ and $\gamma$ are real positive constants.

Recently, for Grushin type operators improved Hardy inequalities were obtained in \cite{ SuYa,YaSuKo},
Hardy inequalities involving the control distance in  \cite{Xi} and Hardy inequalities in half spaces with the degeneracy at the boundary in \cite{Li}.

After the seminal paper \cite{GaLa} by Garofalo and Lanconelli,
where the Hardy inequality  for the Kohn Laplacian on the Heisenberg group was proved, a large amount of work has been devoted to Hardy type inequalities in sub-elliptic settings. For a wide bibliography regarding this topics we directly refer to the paper \cite{DAm2} by D'Ambrosio.


 The proof of our inequalities is based on an approach introduced by Mitidieri in \cite{Mi} 
 for the classical Laplacian. 
 Our results coincide for the particular case of Grushin type operators with the inequalities 
 D'Ambrosio obtained in \cite{DAm}, where he proved that the inequalities are sharp. 
 We derive explicit values for the constants in the 
  inequalities, but are currently not able
 to show its optimality in the general case.
 
The outline of our paper is as follows:
We first introduce the class of operators we consider and formulate several examples. 
In Section \ref{sec_approach} we explain our approach to derive Hardy type inequalities and give a motivation for the
weights  appearing in the inequalities. 
The main results are stated and proved in Section \ref{sec_hardy}.
In the appendix we illustrate the relation between the fundamental solution and Hardy inequalities and 
comment on the difficulties we encounter proving the optimality of the constant in our inequalities.

\section{ $\Delta_\lambda$-Laplacians}\label{sec_def}

Here and in the sequel, we use the following notations. 
We split $\R^N$ into 
$$
\R^N=\R^{N_1}\times\cdots\times\R^{N_k}, \qquad
$$ 
and write 
$$
x=(x^{(1)},\dots,x^{(k)})\in \R^N,\qquad x^{(i)}= (x^{(i)}_1,\dots,x^{(i)}_{N_i}),\qquad i=1,\dots,k.
$$
The degenerate elliptic operators we consider are of the form
$$
\Delta_\lambda=\lambda_1^2\Delta_{x^{(1)}}+\cdots +\lambda_k^2\Delta_{x^{(k)}}, 
$$
where the functions $\lambda_i:\R^{N}\rightarrow\R$ are pairwise different and  $\Delta_{x^{(i)}}$ denotes the classical Laplacian in $\R^{N_i}.$
We denote by $|x|$ the euclidean norm of $x\in\R^m,\ m\in\N,$
and assume the functions $\lambda_i$ are of the form
\begin{align*}
\lambda_1(x)&= 1,\\ 
\lambda_2(x)&=|x^{(1)}|^{\alpha_{21}},\\
\lambda_3(x)&=|x^{(1)}|^{\alpha_{31}}|x^{(2)}|^{\alpha_{32}},\\
&\ \, \vdots \\
\lambda_k(x)&= |x^{(1)}|^{\alpha_{k1}}|x^{(2)}|^{\alpha_{k2}}\cdots|x^{(k-1)}|^{\alpha_{k k-1}},\qquad x\in\R^N,
\end{align*}
where $\alpha_{ij}\geq 0$ for $i=2,\dots,k, j=1,\dots,i-1.$
Setting $\alpha_{ij}=0$ for $j\geq i$ we can write 
\begin{align}\label{lam}
\lambda_i(x)=\prod_{j=1}^{k}|x^{(j)}|^{\alpha_{ij}}, \qquad i=1,\dots,k.
\end{align}

This implies that there exists a \textit{group of dilations} $(\delta_r)_{r>0},$
\begin{eqnarray*}
\delta_r:\R^N\rightarrow\R^N,\quad \delta_r(x)=\delta_r(x^{(1)},\dots,x^{(k)})=(r^{\sigma_1}x^{(1)},\dots,r^{\sigma_k}x^{(k)}),
\end{eqnarray*}
where $1=\sigma_1\leq\sigma_i$ 
such that 
 $\lambda_i$ is $\delta_r$-\emph{homogeneous of degree} $\sigma_i-1$, i.e.,
\begin{eqnarray*}
\label{dilations}\lambda_i(\delta_r(x))=r^{\sigma_i-1}\lambda_i(x),\qquad \forall x\in\R^N,\ r>0,\ i=1,\dots,k,
\end{eqnarray*}
and  the operator $\Delta_\lambda$ is $\delta_r$-homogeneous of degree two, i.e.,
\begin{eqnarray*}
\Delta_\lambda (u(\delta_r(x)))=r^2 (\Delta_\lambda u) (\delta_r(x))\qquad \forall u\in C^\infty (\erre^N).
\end{eqnarray*}

We denote by $Q$ the {\it homogeneous dimension} of $\erre^N$ with respect to the group of dilations $(\delta_r)_{r>0}$, i.e.,
$$
Q:=\sigma_1N_1+\dots+\sigma_k N_k.
$$
$Q$ will play the same role as the dimension $N$ for the classical Laplacian in our Hardy type inequalities. 

For functions $\lambda_i$ of the form \eqref{lam} we find 
\begin{align*}
\sigma_1&= 1,\\
\sigma_2&=1+\sigma_1\alpha_{21},\\ 
\sigma_3&=1+\sigma_1\alpha_{31}+\sigma_2\alpha_{32}, \\
&\ \vdots \\
\sigma_k&= 1+\sigma_1\alpha_{k1}+ \sigma_2\alpha_{k2} +\cdots+\sigma_{k-1}  \alpha_{k k-1}.
\end{align*}

If the functions $\lambda_i$ are smooth, i.e., if the exponents $\alpha_{ji}$ are integers, 
the operator $\Delta_\lambda$ belongs to the general class of
operators studied by H\"ormander in \cite{Ho} and it is hypoelliptic 
(see Remark 1.3, \cite{KoLa}).
The simplest example is the operator 
$$
\partial_{x_1}^2+ |x_1|^{2\alpha}\partial_{x_2}^2,\qquad x=(x_1,x_2)\in\R^2,\ \alpha\in\N,
$$
where $\partial_{x_i}=\frac{\partial}{\partial_{x_i}}, \ i=1,2,$ that Grushin studied in \cite{Gr}. 
He provided a complete characterization of the hypoellipticity for such operators when lower terms with
complex coefficients are added.
For real $\alpha > 0$ the operator is commonly called of \textit{Grushin-type}.
\\

Operators $\Delta_\lambda$ with functions $\lambda_i$ of the form \eqref{lam} 
belong to the class of $\Delta_\lambda$\textit{-Laplacians.} 
Franchi and Lanconelli introduced operators of $\Delta_\lambda$-Laplacian type in 1982 and studied their properties in a series of papers.  
In \cite{FrLa1} they defined a metric associated to these operators that plays the same role as the euclidian metric for the standard Laplacian. 
Using this metric in \cite{FrLa2} and \cite{FrLa3} they extended the classical De Giorgi theorem and obtained Sobolev type embedding theorems 
for such operators.
Recently, adding the assumption that 
the operators are homogeneous of degree two, they were named $\Delta_\lambda$-Laplacians by Kogoj and Lanconelli
in \cite{KoLa}, where existence, non-existence and regularity results for solutions of the semilinear $\Delta_\lambda$-Laplace equation 
were analyzed.
The global well-posedness and longtime behavior of solutions of semilinear degenerate parabolic equations involving 
$\Delta_\lambda$-Laplacians were studied in \cite{KoSo}, and this result was extended in \cite{KoSo2}, where also hyperbolic problems were considered. We finally remark that the $\Delta_\lambda$-Laplacians 
belong to the more general class of $X$\textit{-elliptic operators} introduced in \cite{LaKo}.
For these operators Hardy inequalities of other kind with weights determined by the control distance were proved by Grillo in \cite{Gri}.\\ 

To conclude this section we recall some of the examples 
in our previous paper \cite{KoSo}.
\begin{example}\label{ex1}
Let $\alpha$ be a real positive constant and $k=2$. We consider the Grushin-type operator
\begin{eqnarray*}
\Delta_\lambda=\Delta_{x^\uno}+ |x^\uno|^{2\alpha} \Delta_{x^{(2)}},
\end{eqnarray*}
where $\lambda=(\lambda_1,\lambda_2)$, with $\lambda_1(x)=1$ and $\lambda_2(x) = |x^\uno|^{\alpha},$ $x\in\R^{N_1}\times\R^{N_2}$.
Our group of dilations is
\begin{eqnarray*}
\delta_r\left(x^\uno,x^\due\right)=\left(r x^\uno, r^{\alpha+1} x^\due\right),
\end{eqnarray*}
and the homogenous dimension with respect to 
$(\delta_r)_{r>0}$ is $Q=N_1 +N_2(\alpha+1)$.

More generally, for a given multi-index $\alpha=(\alpha_1,\ldots,\alpha_{k-1})$ with 
real constants $\alpha_i> 0$, $i=1,\ldots,k-1,$ we consider 
\begin{eqnarray*}
\Delta_\lambda = \Delta_{x^\uno} + |x^\uno|^{2\alpha_1} \Delta_{x^{(2)}} +\ldots+ |x^\uno|^{2\alpha_{k-1}} \Delta_{x^\kkk}.
\end{eqnarray*}
 The group of dilations is given by
 \begin{eqnarray*} 
\delta_r\left(x^\uno, \ldots,x^\kkk\right)=\left(r x^\uno,r^{1+\alpha_1}x^{(2)},\ldots, r^{1+\alpha_{k-1}} x^\kkk\right),
\end{eqnarray*}
and the homogeneous dimension is 
$Q=N+\alpha_1N_2+\alpha_2N_3+\cdots+\alpha_{k-1}N_k.$ 
\end{example}

\begin{example}\label{ex2}
For a given multi-index $\alpha=(\alpha_1,\ldots,\alpha_{k-1})$ with 
real constants $\alpha_i> 0$, $i=1,\ldots,k-1,$ we define 
\begin{eqnarray*}
\Delta_\lambda = \Delta_{x^\uno} + |x^\uno|^{2\alpha_1} \Delta_{x^\due} + |x^{(2)}|^{2\alpha_2} \Delta_{x^{(3)}} +\ldots+ |x^{(k-1)}|^{2\alpha_{k-1}} \Delta_{x^\kkk}.
\end{eqnarray*}
Then, in our notation $\lambda=\left(\lambda_1,\ldots,\lambda_k\right)$ with 
\begin{align*} 
\lambda_1 (x)&= 1,\quad
\lambda_i(x)= |x^{(i-1)}|^{\alpha_{i-1}}, \ i=2,\ldots, k,\quad x\in\R^{N_1}\times\cdots\times\R^{N_k},
\end{align*}
and the group of dilations is given by
\begin{eqnarray*} 
\delta_r\left(x^\uno, \ldots,x^\kkk\right)=\left(r^{\sigma_1} x^\uno,\ldots, r^{\sigma_k} x^\kkk\right)
\end{eqnarray*}
with $\sigma_1 =1$ and $\sigma_i =\alpha_{i-1} \sigma_{i-1} +1$ for $i=2,\ldots,k$.
In particular, if 
$\alpha_1=\ldots=\alpha_{k-1} =\alpha$,  the dilations become
\begin{eqnarray*}
\delta_r \left(x^\uno, \ldots, x^\kkk\right) = \left( r x^\uno, r^{\alpha+1} x^{(2)},\ldots, r^{\alpha^{k-1}+\ldots+\alpha+1} x^\kkk\right).
\end{eqnarray*}
\end{example}

\begin{example}\label{ex3}  
Let $\alpha, \beta$ and $\gamma$ be positive real constants. For the operator 
\begin{eqnarray*}
\Delta_\lambda =\Delta_{x^\uno} + |x^\uno|^{2\alpha} \Delta_{x^\due} + |x^\uno|^{2\beta} |x^\due|^{2\gamma} \Delta_{x^{(3)}},
\end{eqnarray*} 
where $\lambda= (\lambda_1,\lambda_2,\lambda_3)$ with
\begin{align*} 
\lambda_1 (x)= 1,\quad
\lambda_2 (x)= |x^{(1)}|^{\alpha},\quad
\lambda_3(x) = |x^{(1)}|^{\beta}|x^{(2)}|^{\gamma},\quad x\in\R^{N_1}\times\R^{N_2}\times\R^{N_3},
\end{align*}
we find the group of dilations
\begin{eqnarray*} 
\delta_r\left(x^\uno,x^\due,x^{(3)}\right)=\left ( r x^\uno, r^{\alpha+1} x^\due, r^{\beta + (\alpha +1)\gamma +1} x^{(3)}\right).
\end{eqnarray*} 
\end{example}

\section{How we approach Hardy-type inequalities}\label{sec_approach}

Our Hardy type inequalities are based on the following approach 
indicated by Mitidieri in \cite{Mi}.

Let $\Omega\subset \R^N,$ $N\geq 3,$ be an open subset and $p>1.$
We assume $u\in C_0^1(\Omega),$ and the vector field $h\in C^1(\Omega;\mathbb{R}^N)$ satisfies $\textnormal{div} h>0.$
The divergence theorem implies
\begin{align*}
\int_{\Omega}|u(x)|^p \textnormal{div} h(x)\,dx = -p\int_{\Omega} |u(x)|^{p-2}u(x)\triangledown u(x)\cdot  h(x)\,dx, 
\end{align*}
where $\cdot$ denotes the inner product in $\R^N.$
Taking the absolute value and using H\"older's inequality we obtain
\begin{align*}
\int_{\Omega}|u(x)|^p \textnormal{div} h(x)dx &=-p\int_{\Omega} |u(x)|^{p-2}u(x)\triangledown u(x)\cdot  h(x)dx\\ 
&\leq  
p\left( \int_{\Omega} |u(x)|^{p} \textnormal{div} h(x)dx\right)^{\frac{p-1}{p}} 
\left( \int_{\Omega}  \frac{|h(x)|^p}{ (\textnormal{div} h(x))^{p-1}}|\triangledown u(x)|^{p}dx\right)^{\frac{1}{p}},
\end{align*}
and it follows that
\begin{align}\label{hardy}
\int_{\Omega}|u(x)|^p \textnormal{div} h(x)dx
\leq 
p^p \int_{\Omega}  \frac{|h(x)|^p}{ (\textnormal{div} h(x))^{p-1}}|\triangledown u(x)|^{p}dx.
\end{align}

If we choose the vector field 
$$
h_\varepsilon(x):=\frac{x}{(|x|^2+\varepsilon)^{\frac{p}{2}}},
$$
where $\varepsilon>0,$ 
then
\begin{align*}
\textnormal{div} h_\varepsilon(x) = \frac{N-p\frac{|x|^2}{|x|^2+\varepsilon}}{(|x|^2+\varepsilon)^{\frac{p}{2}}},\qquad |h_\varepsilon(x)|=\frac{|x|}{(|x|^{2}+\varepsilon)^{\frac{p}{2}}}.
\end{align*}
Assuming that $N>p$ we have $\textnormal{div} h_\varepsilon>0,$ 
and from inequality \eqref{hardy} we obtain
$$
\frac{1}{p^p}\int_\Omega\left(N- p\frac{|x|^2}{|x|^2+\varepsilon}\right)\frac{|u(x)|^p}{(|x|^2+\varepsilon)^{\frac{p}{2}}}\,  dx
\leq 
 \int_{\Omega}\left(N- p\frac{|x|^2}{|x|^2+\varepsilon}\right)^{-(p-1)}\frac{|x|^p}{(|x|^2+\varepsilon)^{\frac{p}{2}}} |\triangledown u(x)|^{p}\, dx.
$$
Taking the limit $\varepsilon$ tends to zero, the classical Hardy inequality 
follows from the dominated convergence theorem,$$
\left( \frac{N-p}{p}\right)^p\int_{\Omega}\frac{|u(x)|^p}{|x|^p} \, dx
\leq 
 \int_{\Omega} |\triangledown u(x)|^{p}\, dx,
$$
 and by a density argument it is satisfied
for all functions $u\in H_0^1(\Omega).$
If the origin $\{0\}$ belongs to the domain $\Omega,$ the constant $\frac{N-p}{p}$ is optimal, but not attained in $H_0^1(\Omega).$
\\

This approach can be generalized to deduce 
Hardy type inequalities for degenerate elliptic operators. 
For the operators $\Delta_\lambda$ with functions $\lambda_i$ of the form \eqref{lam} and 
a function $u$ of class $C^1(\Omega)$ we define
$$
\triangledown_\lambda u:=(\lambda_1\triangledown_{x^{(1)}}u,\dots,\lambda_k\triangledown_{x^{(k)}}u),\qquad
\lambda_i\triangledown_{x^{(i)}}:=(\lambda_i\partial_{x_1^{(i)}},\dots,\lambda_i\partial_{x_{N_i}^{(i)}}),\ i=1,\dots,k.
$$
We will obtain a wide family of Hardy type inequalities, that include as particular cases inequalities of the form
\begin{align}
\left(\frac{Q-p}{p}\right)^p\int_\Omega\frac{|u(x)|^{p}}{[[x]]_\lambda^p}dx&\leq 
\int_\Omega\psi(x)|\triangledown_\lambda u(x)|^pdx,\label{hardy1}\\
\left(\frac{Q-p}{p}\right)^p\int_\Omega\varphi(x)\frac{|u(x)|^{p}}{[[x]]_\lambda^p}dx&\leq 
\int_\Omega|\triangledown_\lambda u(x)|^p dx,\label{hardy2}
\end{align}
where $Q$ is the homogeneous dimension,
and $\varphi$ and $\psi$ are suitable weight functions. Moreover, $[[\cdot]]_\lambda$ is a homogeneous norm
that replaces the euclidean norm in the classical Hardy inequality.
\\

We introduce the following notation. For a vector field $h$ of class $C^1(\Omega;\R^N)$  
we define
$$
\textnormal{div}_\lambda h :=\sum_{i=1}^k\lambda_i\textnormal{div}_{x^{(i)}} h,\qquad
\textnormal{div}_{x^{(i)}} h :=\sum_{j=1}^{N_i}\partial_{x_j^{(i)}} h.
$$
The subsequent lemma 
follows from the divergence theorem and can be shown similarly  as 
 inequality \eqref{hardy}.
See also Theorem 3.5 in \cite{DAm} for the particular case of Grushin-type operators.

\begin{lemma}\label{lem0}
Let $h\in C^1(\Omega;\R^N)$ be such that $\textnormal{div}_{\lambda} h\geq0.$
Then, for every $p>1$ and $u\in C_0^1(\Omega)$ such that 
$\frac{|h|}{(\textnormal{div}_{\lambda} h)^{\frac{p}{p-1}}}|\triangledown_{\lambda}u|\in L^p(\Omega)$ we have 
$$
\int_\Omega |u(x)|^p \textnormal{div}_{\lambda} h(x)\, dx \leq 
p^p \int_\Omega \frac{|h(x)|^p}{(\textnormal{div}_{\lambda} h(x))^{p-1}} |\triangledown_{\lambda} u(x)|^p\, dx. 
$$
\end{lemma}

\begin{proof}
We define 
$$ \sigma:=
\left(
\begin{matrix}
  I_{1} & 0 & \cdots  &0   \\
  0 & \lambda_2 I_{2} &  & \vdots\\
  \vdots&&\ddots&0\\[1.2ex]
  0 & \cdots&0& \lambda_k I_{k}
 \end{matrix}\right),
$$
where $I_{i}$ denotes the identity matrix in $\R^{N_i},$ $i=1,\dots,k.$
The divergence theorem implies
\begin{align*}
0&=\int_{\partial\Omega}|u|^p h\cdot \sigma \nu\, d\zeta=\int_\Omega \textnormal{div}_{\lambda}(|u|^p h)\,dx
=\int_\Omega p|u|^{p-2}u\triangledown_{\lambda} u \cdot h\, dx+\int_\Omega |u|^p \textnormal{div}_{\lambda}h\, dx,
\end{align*}
where $\nu$ denotes the outward unit normal at $\zeta\in\partial \Omega.$ Applying H\"older's inequality we obtain
\begin{align*}
\int_\Omega |u|^p \textnormal{div}_{\lambda}h\, dx&=-\int_\Omega p|u|^{p-2}u\triangledown_{\lambda} u \cdot h\, dx\leq
\int_\Omega p|u|^{p-1}|\triangledown_{\lambda} u| |h| \,dx\\
&\leq p\left( \int_\Omega |u|^{p}  (\textnormal{div}_{\lambda}h+\epsilon) \, dx\right)^{\frac{p-1}{p}} \left(\int_\Omega \frac{|h|^p}{ (\textnormal{div}_{\lambda}h +\epsilon)^{p-1}} |\triangledown_{\lambda} u|^p \,dx\right)^{\frac{1}{p}},
\end{align*}
and consequently,
\begin{align*}
\frac{\int_\Omega |u|^p \textnormal{div}_{\lambda}h\, dx}{\left( \int_\Omega |u|^{p}  (\textnormal{div}_{\lambda}h+\epsilon) \, dx\right)^{\frac{p-1}{p}}}&
\leq p\left(\int_\Omega \frac{|h|^p}{ (\textnormal{div}_{\lambda}h +\epsilon)^{p-1}} |\triangledown_{\lambda} u|^p \,dx\right)^{\frac{1}{p}}.
\end{align*}
The statement of the lemma now follows from the dominated convergence theorem.
\end{proof}

To illustrate our approach we first consider Hardy type inequalities of the form \eqref{hardy1}, i.e., 
\begin{align*}
\left(\frac{Q-p}{p}\right)^p\int_\Omega\frac{|u(x)|^{p}}{[[x]]_\lambda^p}\,dx&\leq 
\int_\Omega \psi(x)|\triangledown_\lambda u(x)|^p\,dx,
\end{align*}
with a certain weight function $\psi$ and homogeneous norm $[[ \cdot ]]_\lambda.$

Motivated by Lemma \ref{lem0} we look for a function $h$ satisfying  
$$
\textnormal{div}_\lambda h(x)=\frac{Q-p}{[[x]]_\lambda^p}.
$$
If we choose 
$$
h(x)=\frac{1}{[[x]]_\lambda^p}\left(\frac{\sigma_1 x^{(1)}}{\lambda_1(x)},\dots,\frac{\sigma_k x^{(k)}}{\lambda_k(x)} \right),
$$ 
and since $\lambda_i$ does not depend on $x^{(i)}$ we obtain
$$
\textnormal{div}_\lambda h(x)=\frac{Q}{[[x]]_\lambda^p}-p\frac{1}{[[x]]_\lambda^{p+1}}\sum_{i=1}^k\sigma_ix^{(i)}\cdot \triangledown_{x^{(i)}}(\|x\|_\lambda).
$$
Consequently, the homogeneous norm $[[\cdot]]_\lambda$ should fulfill the relation 
\begin{align}\label{p1}
\sum_{i=1}^k\sigma_ix^{(i)}\cdot\triangledown_{x^{(i)}}([[x]]_\lambda)=[[x]]_\lambda.
\end{align}
On the other hand, computing the norm of $h$ we obtain
$$
|h(x)|^2=\frac{1}{[[x]]_\lambda^{2p}}\frac{1}{\prod_{i=1}^k\lambda_i(x)^2}
\left(\prod_{j\neq 1}\lambda_j(x)^2 \sigma_1^2 |x^{(1)}|^2 + \dots  +  \prod_{j\neq k}\lambda_j(x)^2 \sigma_k^2 |x^{(k)}|^2
 \right),
$$
which motivates to consider the homogeneous norm
\begin{align}\label{p2}
[[x]]_\lambda=
\left(\prod_{j\neq 1}\lambda_j(x)^2 \sigma_1^2 |x^{(1)}|^2 + \dots  +  \prod_{j\neq k}\lambda_j(x)^2 \sigma_k^2|x^{(k)}|^2
 \right)^{\frac{1}{2(1+\sum_{i=1}^k(\sigma_i-1))}}.
\end{align}
The exponent is determined by requiring  $[[\cdot]]_\lambda$ to be $\delta_r$-homogeneous of degree one.
Since the functions $\lambda_i$ are of the form \eqref{lam},
the relation \eqref{p1} is satisfied.

\section{Hardy Inequalities for $\Delta_\lambda$-Laplacians}\label{sec_hardy}

\subsection{Our homogeneous norms}\label{sec_semi}

We recall that
$
\Delta_\lambda=\lambda_1^2\Delta_{x^{(1)}}+\cdots +\lambda_k^2\Delta_{x^{(k)}}
$
with functions $\lambda_i$  of the form 
\begin{align*}
\lambda_i(x)=\prod_{j=1}^{k}|x^{(j)}|^{\alpha_{ij}}, \qquad i=1,\dots,k,
\end{align*}
which are 
$\delta_r$-homogeneous of degree $\sigma_i-1$ 
with respect to a group of dilations 
$$
\delta_r(x)=(r^{\sigma_1}x^{(1)},\dots, r^{\sigma_k}x^{(k)}),\qquad x\in\R^N,\  r>0.
$$
Using our previous notations  follow the relations 
\begin{align*}
\sum_{j=1}^k\alpha_{ij}\sigma_j&=\sigma_i-1,&
\prod_{i=1}^k\lambda_i(x) &=\prod_{j=1}^k|x^{(j)}|^{\sum_{i=1}^k\alpha_{ij}}.
\end{align*}

\begin{definition}

We define the homogenous norm $[[\cdot]]_\lambda$ associated to the $\Delta_\lambda$-Laplacian by relation \eqref{p2}, 
\begin{align*}
[[x]]_\lambda:=\left(\prod_{i \neq 1}\lambda_i(x)^2\sigma_1^2|x^{(1)}|^2+\cdots+ \prod_{i\neq k}\lambda_i(x)^2\sigma_k^2|x^{(k)}|^2\right)^{\frac{1}{2(1+\sum_{i=1}^k(\sigma_i-1))}},\qquad x\in\R^N.
\end{align*}
\end{definition}

Under our hypotheses $[[\cdot]]_\lambda$  can be written as
\begin{align*}
[[x]]_\lambda&=\left(\prod_{j= 1}^k   |x^{(j)}|^{\sum_{i\neq 1}2\alpha_{ij}}  \sigma_1^2|x^{(1)}|^2+\cdots+ 
\prod_{j= 1}^k   |x^{(j)}|^{\sum_{i\neq k}2\alpha_{ij}} \sigma_k^2|x^{(k)}|^2\right)^{\frac{1}{2(1+\sum_{i=1}^k(\sigma_i-1))}}.
\end{align*}

We compute the homogeneous norm $[[\cdot]]_\lambda$ for some of the operators in our previous examples.
\begin{itemize}

\item For Grushin-type operators 
$$
\Delta_\lambda=\Delta_x+|x|^{2\alpha}\Delta_y, \qquad
 (x,y)\in\R^{N_1}\times\R^{N_2},
$$ 
where the constant $\alpha$ is non-negative, the definition leads to 
the same distance from the origin that D'Ambrosio considered in \cite{DAm},
$$
[[(x,y)]]_\lambda=\left(|x|^{2(1+\alpha)}+(1+\alpha)^2 |y|^2\right)^{\frac{1}{2(1+\alpha)}}.
$$

\item For operators of the form 
$$
\Delta_\lambda=\Delta_x+|x|^{2\alpha}\Delta_y+|x|^{2\beta}\Delta_z,\qquad (x,y,z)\in\R^{N_1}\times\R^{N_2}\times\R^{N_3},
$$ 
with non-negative constants $\alpha$ and $\beta,$ we obtain 
$$
[[(x,y,z)]]_\lambda=\left(|x|^{2(1+\alpha+\beta)}+(1+\alpha)^2|x|^{2\beta} |y|^2
+(1+\beta)^2|x|^{2\alpha} |z|^2\right)^{\frac{1}{2(1+\alpha+\beta)}}.
$$

\item For $\Delta_\lambda$-Laplacians of the form 
 $$
 \Delta_\lambda=\Delta_x+|x|^{2\alpha}\Delta_y+|x|^{2\beta}|y|^{2\gamma}\Delta_z,\qquad
 (x,y,z)\in\R^{N_1}\times\R^{N_2}\times\R^{N_3},
 $$ 
 where the constants $\alpha, \beta$ and $\gamma$ are non-negative,
 we get
\begin{align*} 
[[(x,y,z)]]_\lambda 
= \left(|y|^{2\gamma}|x|^{2(1+\alpha+\beta)}+(1+\alpha)^2|x|^{2\beta}|y|^{2(1+\gamma)}
+(1+\mu)^2|x|^{2\alpha} |z|^2\right)^{\frac{1}{2(1+\alpha+\mu)}},
\end{align*}
where $\mu=\beta+(1+\alpha)\gamma.$

\end{itemize}

\begin{proposition}\label{prop}
 Our homogeneous norm $[[\cdot]]_\lambda$ satisfies the following properties:
 
 \begin{itemize}

\item[(1)] 
It is $\delta_r$-homogeneous of degree one, i.e.,
$$
[[\delta_r(x)]]_\lambda=r[[x]]_\lambda.
$$

\item[(2)]
It fulfills the relation 
$$
\sum_{i=1}^k\sigma_i\left(x^{(i)}\cdot \triangledown_{x^{(i)}}\right)[[x]]_\lambda=[[x]]_\lambda.
$$

\end{itemize}

\end{proposition}

\begin{proof}
$(1)$\ 
Let $x\in\R^N.$ The homogeneity of the functions $\lambda_i$ implies that
\begin{align*}
&\  [[\delta_r(x)]]_\lambda \\
=&\ \left(\prod_{i \neq 1}(\lambda_i(\delta_r(x)))^2\sigma_1^2|r^{\sigma_1}x^{(1)}|^2+\cdots+ 
\prod_{i\neq k}(\lambda_i(\delta_r(x)))^2\sigma_k^2|r^{\sigma_k}x^{(k)}|^2\right)^{\frac{1}{2(1+\sum_{i=1}^k(\sigma_i-1))}}\\
=&\ \left(\prod_{i \neq 1}r^{2\sigma_1}r^{2(\sigma_i-1)}(\lambda_i(x))^2\sigma_1^2|x^{(1)}|^2+\cdots+ 
\prod_{i\neq k}r^{2\sigma_k}r^{2(\sigma_i-1)}(\lambda_i(x))^2\sigma_k^2|x^{(k)}|^2\right)^{\frac{1}{2(1+\sum_{i=1}^k(\sigma_i-1))}}\\
=&\ \left(r^{2+\sum_{i=1}^k2(\sigma_i-1)}\prod_{i \neq 1}(\lambda_i(x))^2\sigma_1^2|x^{(1)}|^2+\cdots+ \prod_{i\neq k}(\lambda_i(x))^2\sigma_k^2|x^{(k)}|^2\right)^{\frac{1}{2(1+\sum_{i=1}^k(\sigma_i-1))}}
= r[[x]]_\lambda.
\end{align*}
\\
$(2)$\ 
We observe 
\begin{align*}
&\ x^{(l)}\cdot \triangledown_{x^{(l)}}[[x]]_\lambda\\
=&\ \frac{1}{2(1+\sum_{i=1}^k(\sigma_i-1))}\left(\prod_{i \neq 1}(\lambda_i(x))^2\sigma_1^2|x^{(1)}|^2+\cdots+ \prod_{i\neq k}(\lambda_i(x))^2\sigma_k^2|x^{(k)}|^2\right)^{\frac{1}{2(1+\sum_{i=1}^k(\sigma_i-1))}-1}\\
&\ \left( (2\sum_{j\neq1}\alpha_{jl})\prod_{i \neq 1}(\lambda_i(x))^2\sigma_1^2|x^{(1)}|^2+\cdots+ (2\sum_{j\neq k}\alpha_{jl}) \prod_{i\neq k}(\lambda_i(x))^2\sigma_k^2|x^{(k)}|^2+
2\prod_{i \neq l}(\lambda_i(x))^2\sigma_l^2|x^{(l)}|^2 \right),
\end{align*}
and using the relation $\sum_{l=1}^k\sigma_l\alpha_{jl}=\sigma_j-1$ it follows that 
\begin{align*}
&\ \sum_{l=1}^k\sigma_l\left(x^{(l)}\cdot \triangledown_{x^{(l)}}\right)[[x]]_\lambda\\
=&\  \frac{1}{2(1+\sum_{i=1}^k(\sigma_i-1))}\left(\prod_{i \neq 1}(\lambda_i(x))^2\sigma_1^2|x^{(1)}|^2+\cdots+ \prod_{i\neq k}(\lambda_i(x))^2\sigma_k^2|x^{(k)}|^2\right)^{\frac{1}{2(1+\sum_{i=1}^k(\sigma_i-1))}-1}\\
&\ 2\left( ((\sum_{j\neq1}\sum_{l=1}^k\sigma_l\alpha_{jl})+\sigma_1)\prod_{i \neq 1}(\lambda_i(x))^2\sigma_1^2|x^{(1)}|^2+\cdots+ ((\sum_{j\neq k}\sum_{l=1}^k\sigma_l\alpha_{jl})+\sigma_k) \prod_{i\neq k}(\lambda_i(x))^2\sigma_k^2|x^{(k)}|^2 \right) \\
=&\ [[x]]_\lambda.
\end{align*}
\end{proof}

\subsection{Main results}

We denote by $\mathring{W}_\lambda^{1,p}(\Omega)$ 
the closure of $C_0^1(\Omega)$ with respect to the norm
$$
\|u\|_{\mathring{W}_\lambda^{1,p}(\Omega)}:=\left(\int_\Omega|\triangledown_\lambda u(x)|^p dx\right)^{\frac{1}{p}},
$$
 and for $\psi\in L^1_{loc}(\Omega)$ such that $\psi>0$ a.e. in $\Omega$ 
we define the space $\mathring{W}_\lambda^{1,p}(\Omega,\psi)$ as
 the closure of $C_0^1(\Omega)$ with respect to the norm 
$$
\|u\|_{\mathring{W}_\lambda^{1,p}(\Omega;\psi)}:=\left(\int_\Omega|\triangledown_\lambda u(x)|^p \psi(x) dx\right)^{\frac{1}{p}}.
$$

\begin{theorem}\label{thm_semi}
Let $p>1$ and $\mu_1,\dots,\mu_k,\ s\in\R$ be such that
$s<N_1+\mu_1$ and 
\begin{align}\label{cond1}
-p\min\{\alpha_{1i},\dots\alpha_{ki},1 \} +s<N_i+\mu_i \qquad i=1,\dots,k.
\end{align}
Then, for every $u\in \mathring{W}_\lambda^{1,p}(\Omega, \psi),$ we have 
$$
\left(\frac{Q-s+\sum_{i=1}^k\sigma_i\mu_i}{p}\right)^p\int_\Omega\frac{\prod_{i=1}^k|x^{(i)}|^{\mu_i}}{[[x]]_\lambda^s}|u(x)|^p\ dx\leq
\int_\Omega\psi(x)\left|\triangledown_\lambda u(x)\right|^p\ dx,
$$
where 
$ \psi(x)=\frac{[[x]]_\lambda^{p(1+\sum_{i=1}^k(\sigma_i-1))-s}}{\prod_{i=1}^k  |x^{(i)}|^{p(\sum_{j=1}^k\alpha_{ji})- \mu_i} }.$ 

In particular, for $s=p$ and $\mu_1=\cdots=\mu_k=0$ we get
$$
\left(\frac{Q-p}{p}\right)^p\int_\Omega\frac{|u(x)|^p}{[[x]]_\lambda^p}\ dx\leq
\int_\Omega\frac{[[x]]_\lambda^{\sum_{i=1}^k(\sigma_i-1)}}{\prod_{i=1}^k\lambda_i(x)^p}\left|\triangledown_\lambda u(x)\right|^p\ dx,
$$
and choosing $s=p(1+\sum_{i=1}^k(\sigma_i-1))$ and $\mu_i=p\sum_{j=1}^k\alpha_{ji}$ we obtain
$$
\left(\frac{Q-p}{p}\right)^p\int_\Omega\frac{\prod_{i=1}^k\lambda_i(x)^p}{[[x]]_\lambda^{p(1+\sum_{i=1}^k(\sigma_i-1))}}|u(x)|^p\ dx
\leq\int_\Omega \left|\triangledown_\lambda u(x)\right|^p\ dx.
$$
\end{theorem}

\begin{proof}
We deduce the inequalities from Lemma \ref{lem0}. To this end for $\varepsilon>0$ we 
define
\begin{align*}
\lambda^\varepsilon&:=(\lambda_1^\varepsilon,\dots,\lambda_k^\varepsilon),\qquad 
\lambda_i^\varepsilon(x):=\prod_{j=1}^k\left(|x^{(j)}|^2+\varepsilon\right)^{\frac{\alpha_{ij}}{2}}, \quad i=1,\dots, k,\\
[[x]]_{\varepsilon,\lambda}&:=\left(\sum_{j=1}^k\Big(\prod_{i\neq j}\lambda_i^\varepsilon(x)^2\sigma_j^2|x^{(j)}|^2\Big)\right)^{\frac{1}{2(1+\sum_{i=1}^k(\sigma_i-1))}}
\end{align*}
and consider the function
$$
h_\varepsilon(x):=\frac{\prod_{i=1}^k|x^{(i)}|^{\mu_i}}{[[x]]_{\varepsilon,\lambda}^s}
\left(\frac{\sigma_1x^{(1)}}{\lambda^\varepsilon_1(x)},\dots, \frac{\sigma_k x^{(k)}}{\lambda^\varepsilon_k(x)}\right).
$$
We obtain
\begin{align*}
\textnormal{div}_{\lambda} h_\varepsilon(x)=&\ 
\sum_{i=1}^k\frac{\lambda_i(x)}{\lambda_i^\varepsilon(x)}\triangledown_{x^{(i)}}\cdot \left(\frac{\prod_{i=1}^k|x^{(i)}|^{\mu_i}}{[[x]]_{\varepsilon,\lambda}^s} \sigma_i x^{(i)} \right)\\
=&\frac{\prod_{i=1}^k|x^{(i)}|^{\mu_i}}{[[x]]_{\varepsilon,\lambda}^s} 
\left( \sum_{i=1}^k\frac{\lambda_i(x)}{\lambda_i^\varepsilon(x)}\left(N_i\sigma_i+\sigma_i\mu_i-s \frac{1}{[[x]]_{\varepsilon,\lambda}}\sigma_i x^{(i)} \cdot\triangledown_{x^{(i)}}([[x]]_{\varepsilon,\lambda})\right)\right)\\
=&\frac{\prod_{i=1}^k|x^{(i)}|^{\mu_i}}{[[x]]_{\varepsilon,\lambda}^s} c_\varepsilon(x),
\end{align*}
where 
$$
c_\varepsilon(x):=\sum_{i=1}^k\frac{\lambda_i(x)}{\lambda_i^\varepsilon(x)}\left(N_i\sigma_i+\sigma_i\mu_i-s \frac{1}{[[x]]_{\varepsilon,\lambda}}\sigma_i x^{(i)} \cdot\triangledown_{x^{(i)}}([[x]]_{\varepsilon,\lambda})\right).
$$
Using Proposition \ref{prop} we observe that
\begin{align*}
\lim_{\varepsilon\rightarrow 0} c_\varepsilon(x)=\sum_{i=1}^k\left(N_i\sigma_i+\sigma_i\mu_i-s \right)= Q-s+\sum_{i=1}^k\sigma_i\mu_i,
\end{align*}
which is positive by our hypothesis.
Moreover, there exist positive constants $\alpha_1$ and $\alpha_2$ such that
\begin{align}\label{bound}
0<\alpha_1\leq c_\varepsilon(x)\leq \alpha_2<\infty\qquad  \forall x\in\Omega.
\end{align}
Indeed, we compute
\begin{align*}
&\ x^{(l)}\cdot \triangledown_{x^{(l)}}[[x]]_{\varepsilon,\lambda}\\
=&\ \frac{1}{2(1+\sum_{i=1}^k(\sigma_i-1))}\left(\prod_{i \neq 1}(\lambda_i^\varepsilon(x))^2\sigma_1^2|x^{(1)}|^2+\cdots+ \prod_{i\neq k}(\lambda_i^\varepsilon(x))^2\sigma_k^2|x^{(k)}|^2\right)^{\frac{1}{2(1+\sum_{i=1}^k(\sigma_i-1))}-1}\\
=&\ \Big\{ \frac{|x^{(l)}|^2}{|x^{(l)}|^2+\varepsilon}\Big[(2\sum_{j\neq1}\alpha_{jl})\prod_{i \neq 1}(\lambda_i^\varepsilon(x))^2\sigma_1^2|x^{(1)}|^2+\cdots+ (2\sum_{j\neq k}\alpha_{jl}) \prod_{i\neq k}(\lambda_i^\varepsilon(x))^2\sigma_k^2|x^{(k)}|^2\Big]\\
& +2\prod_{i \neq l}(\lambda_i^\varepsilon(x))^2\sigma_l^2|x^{(l)}|^2 \Big\},
\end{align*}
and consequently, using the relation $\sum_{l=1}^k\sigma_l\alpha_{jl}=\sigma_j-1$ it follows that 
\begin{align*}
c_\varepsilon(x)=&\sum_{l=1}^k\frac{\lambda_l(x)}{\lambda_l^\varepsilon(x)}\left(N_l\sigma_l+\sigma_l\mu_l\right)
-s\frac{1}{[[x]]_{\varepsilon,\lambda}}\sum_{l=1}^k\frac{\lambda_l(x)}{\lambda_l^\varepsilon(x)} \sigma_l 
x^{(l)} \cdot\triangledown_{x^{(l)}}([[x]]_{\varepsilon,\lambda})\\
&\sum_{l=1}^k\frac{\lambda_l(x)}{\lambda_l^\varepsilon(x)}\left(N_l\sigma_l+\sigma_l\mu_l\right)
-s\frac{1}{[[x]]_{\varepsilon,\lambda}}[[x]]_{\varepsilon,\lambda}^{1-2(1+\sum_{i=1}^k(\sigma_i-1))}\sum_{l=1}^k \sigma_l 
x^{(l)} \cdot\triangledown_{x^{(l)}}([[x]]_{\varepsilon,\lambda})\\
\geq &\sum_{l=1}^k\frac{\lambda_l(x)}{\lambda_l^\varepsilon(x)}\left(N_l\sigma_l+\sigma_l\mu_l\right)-
s\frac{[[x]]_{\varepsilon,\lambda}^{-2(1+\sum_{i=1}^k(\sigma_i-1))}}{(1+\sum_{i=1}^k(\sigma_i-1))}\cdot  \\
&\ \Big\{ ((\sum_{j\neq1}\sum_{l=1}^k\sigma_l\alpha_{jl})+\sigma_1)\prod_{i \neq 1}(\lambda_i^\varepsilon(x))^2\sigma_1^2|x^{(1)}|^2+\cdots+ ((\sum_{j\neq k}\sum_{l=1}^k\sigma_l\alpha_{jl}) +\sigma_k)\prod_{i\neq k}(\lambda_i^\varepsilon(x))^2\sigma_k^2|x^{(k)}|^2  \Big\}\\
=&\sum_{l=1}^k\frac{\lambda_l(x)}{\lambda_l^\varepsilon(x)}\left(N_l\sigma_l+\sigma_l\mu_l\right)-
s\geq N_1+\mu_1-s>0.
\end{align*}
On the other hand,
\begin{align*}
c_\varepsilon(x)=&\sum_{l=1}^k\frac{\lambda_l(x)}{\lambda_l^\varepsilon(x)}\left(N_l\sigma_l+\sigma_l\mu_l-s \frac{1}{[[x]]_{\varepsilon,\lambda}}\sigma_l 
x^{(l)} \cdot\triangledown_{x^{(l)}}([[x]]_{\varepsilon,\lambda})\right)\\
\leq & \sum_{l=1}^k\left(N_l\sigma_l+\sigma_l\mu_l\right)<\infty,
\end{align*}
which concludes the proof of property \eqref{bound}.

Moreover, we compute 
\begin{align*}
| h_\varepsilon(x) | & =\ \frac{\prod_{i=1}^k|x^{(i)}|^{\mu_i}}{[[x]]_{\varepsilon,\lambda}^s}
\left(\sum_{i=1}^k\frac{\sigma_i^2|x^{(i)}|^2}{\lambda^\varepsilon_i(x)^2}\right)^{\frac{1}{2}}
\\
&=\ \frac{\prod_{i=1}^k|x^{(i)}|^{\mu_i}}{[[x]]_{\varepsilon,\lambda}^s}
\frac{\left( \sum_{i=1}^k  \prod_{j\neq i} \lambda^\varepsilon_j(x)^2 \sigma_i^2|x^{(i)}|^2   \right)^{\frac{1}{2}}}{\prod_{i=1}^k \lambda^\varepsilon_i(x)}\\
&=
\frac{\prod_{i=1}^k|x^{(i)}|^{\mu_i}[[x]]_{\varepsilon,\lambda}^{(1+\sum_{i=1}^k(\sigma_i-1))-s} 
 }{\prod_{i=1}^k \lambda^\varepsilon_i(x)},
\end{align*}
and Lemma \ref{lem0} applied to $h_\varepsilon$ yields
\begin{align*} 
&\frac{1}{p^p}\int_\Omega c_\varepsilon(x) \frac{\prod_{i=1}^k|x^{(i)}|^{\mu_i}}{[[x]]_{\varepsilon,\lambda}^s} |u(x)|^p dx\\
\leq &
\ \int_\Omega \frac{1}{c_\varepsilon(x)^{(p-1)}}\frac{\prod_{i=1}^k|x^{(i)}|^{\mu_i}}{[[x]]_{\varepsilon,\lambda}^{s}} 
\left(\sum_{i=1}^k\frac{\sigma_i^2|x^{(i)}|^2}{\lambda^\varepsilon_i(x)^2}\right)^{\frac{p}{2}}|\triangledown_{\lambda} u(x)|^p\ dx,\\
\leq &
\ \frac{1}{\alpha_1^{(p-1)}}\int_\Omega \frac{\prod_{i=1}^k|x^{(i)}|^{\mu_i}}{[[x]]_{\lambda}^{s}} 
\left(\sum_{i=1}^k\frac{\sigma_i^2|x^{(i)}|^2}{\lambda_i(x)^2}\right)^{\frac{p}{2}}|\triangledown_{\lambda} u(x)|^p\ dx\\
=&\ \frac{1}{\alpha_1^{(p-1)}}\int_\Omega\psi(x)|\triangledown_{\lambda} u(x)|^p\ dx.
\end{align*}
Since 
$$
\lim_{\varepsilon\rightarrow0}c_{\varepsilon}(x)=Q+\sum_{i=1}^k\sigma_i\mu_i-s,
$$
the theorem now follows
from the dominated convergence theorem by taking the limit $\varepsilon$ tends to zero.
\end{proof}

\begin{remark}
The first condition on the exponents in Theorem 
\ref{thm_semi} allows to derive the uniform estimates for $c_\epsilon(x)$ in the proof, while the condition
\eqref{cond1} ensures that 
$\psi$ belongs to $L^1_{loc}(\Omega).$
\end{remark}

We formulated a very general family of Hardy-type inequalities, the parameters allow to adjust the weights and to move 
them from one side of the inequality to the other. Particular choices lead to inequalities of the form \eqref{hardy1} or \eqref{hardy2}.

\begin{remark}
For Grushin-type operators  $\Delta_\lambda=\Delta_x+|x|^{2\alpha}\Delta_y,$ 
$\alpha\geq 0, (x,y)\in\R^{N_1}\times\R^{N_2},$ we recover the Hardy inequalities of Theorem 3.1 in \cite{DAm}, 
where it was proved that the constants are optimal. 
\end{remark}

For the convenience of the reader we first
formulated Hardy type inequalities for the particular case of our homogeneous norms $[[\cdot]]_\lambda.$
We now generalize Theorem \ref{thm_semi} 
and consider homogeneous distances from the origin $\|\cdot \|_\lambda$
that satisfy the relation 
$$
\sum_{j=1}^k\sigma_j\left( x^{(j)}\cdot\triangledown_{x^{(j)}} \right) \|x\|_{\lambda} = \|x\|_{\lambda},\qquad x\in\R^N.
$$
For instance, we could choose
\begin{align}
&&\|x\|_{\lambda}&:=\left(\sum_{j=1}^k |x^{(j)}|^{2\prod_{i\neq j}\sigma_i}\right)^{\frac{1}{2\prod_{i=1}^k\sigma_i}},&& x\in\R^N,\label{dist1}
\\
\textnormal{or}&&\|x\|_{\lambda}&:=\left(\sum_{j=1}^k(\sigma_j |x^{(j)}|)^{2\prod_{i\neq j}\sigma_i}\right)^{\frac{1}{2\prod_{i=1}^k\sigma_i}},&& x\in\R^N.
\label{dist2}
\end{align}

\begin{remark}
For Grushin-type operators the second distance $\|\cdot\|_{\lambda}$ coincides with our homogeneous norm $[[\cdot]]_\lambda$ and with the distance considered 
by D'Ambrosio in \cite{DAm}.
\end{remark}

We compute the first of the homogeneous distances for our previous examples.
\begin{itemize}

\item For operators of the form 
$$
\Delta_\lambda=\Delta_x+|x|^{2\alpha}\Delta_y+|x|^{2\beta}\Delta_z,\qquad (x,y,z)\in\R^{N_1}\times\R^{N_2}\times\R^{N_3},
$$ 
with non-negative constants $\alpha$ and $\beta,$ we obtain 
$$
\|(x,y,z)\|_\lambda=\left(|x|^{2(1+\alpha)(1+\beta)}+ |y|^{2(1+\beta)}
+|z|^{2(1+\alpha)}\right)^{\frac{1}{2(1+\alpha)(1+\beta)}}.
$$

\item For $\Delta_\lambda$-Laplacians of the form 
 $$
 \Delta_\lambda=\Delta_x+|x|^{2\alpha}\Delta_y+|x|^{2\beta}|y|^{2\gamma}\Delta_z,\qquad
 (x,y,z)\in\R^{N_1}\times\R^{N_2}\times\R^{N_3},
 $$ 
 where the constants $\alpha, \beta$ and $\gamma$ are non-negative,
 we get
\begin{align*} 
\|(x,y,z)\|_\lambda 
= \left(|x|^{2(1+\alpha)(1+\mu)}+|y|^{2(1+\mu)}
+|z|^{2(1+\alpha)}\right)^{\frac{1}{2(1+\alpha)(1+\mu)}},
\end{align*}
where $\mu=\beta+(1+\alpha)\gamma.$

\end{itemize}

\begin{theorem}\label{thm_dist}
Let $p>1$ and $\mu_1,\dots,\mu_k, s,t\in\R$ be such that  
$s+t<N_1+\mu_1$ and
\begin{align}\label{cond2}
-p\min\{\alpha_{1i},\dots\alpha_{ki},1 \} +s+\frac{t}{\sigma_i}<N_i+\mu_i \qquad i=1,\dots,k.
\end{align}

Then, for every $u\in\mathring{W}_\lambda^{1,p}(\Omega,\psi)$ we have 
\begin{align*}
\left(\frac{Q-s-t+\sum_{i=1}^k\sigma_i\mu_i}{p}\right)^p\int_\Omega\frac{\prod_{i=1}^k|x^{(i)}|^{\mu_i}}{\|x\|_{\lambda}^t[[x]]_\lambda^s}|u(x)|^p\ dx
& \leq\ \int_\Omega
\psi(x)\left|\triangledown_\lambda u(x)\right|^p\ dx,
\end{align*}
where $\psi(x)=\frac{\prod_{i=1}^k|x^{(i)}|^{ \mu_i-p\sum_{j=1}^k\alpha_{ji}}}{\|x\|_{\lambda}^t [[x]]_\lambda^{s-p(1+\sum_{i=1}^k(\sigma_i-1))}},$
and $\|\cdot\|_{\lambda}$ denotes the homogeneous norm \eqref{dist1} or \eqref{dist2}.

In particular, for $s=0,$ $\mu_i=0$ and $t=p$ we obtain
\begin{align*}
\left(\frac{Q-p}{p}\right)^p\int_\Omega\frac{|u(x)|^p}{\|x\|_{\lambda}^p}\ dx&\leq
\int_\Omega\frac{[[x]]_\lambda^{p(1+\sum_{i=1}^k(\sigma_i-1))}}{\|x\|_{\lambda}^p\prod_{j=1}^k\lambda_j(x)^p}
\left|\triangledown_\lambda u(x)\right|^p\ dx.
\end{align*}

For $t=0$ we recover the Hardy inequalities in Theorem \ref{thm_semi} with our homogeneous norms $[[\cdot]]_\lambda.$
\end{theorem}

\begin{proof} 
We prove the statement for the homogeneous norm \eqref{dist1}. The result for the
distance \eqref{dist2} follows analogously.
We deduce the inequalities from Lemma \ref{lem0}. To this end we define the function
$$
h_\varepsilon(x):=\frac{\prod_{i=1}^k|x^{(i)}|^{\mu_i}}{\|x\|_{\varepsilon,\lambda}^t [[x]]_{\varepsilon,\lambda}^s}
\left(\frac{\sigma_1x^{(1)}}{\lambda_1^\varepsilon(x)},\dots, \frac{\sigma_kx^{(k)}}{\lambda_k^\varepsilon(x)}\right),
$$
where $||\cdot||_{\varepsilon,\lambda}$ is a smooth approximation of $||\cdot||_\lambda,$ 
\begin{align*}
||x||_{\varepsilon,\lambda}&=\left(\sum_{j=1}^k (|x^{(j)}|^2+\varepsilon)^{\prod_{i\neq j}\sigma_i}\right)^{\frac{1}{2\prod_{i=1}^k\sigma_i}}.
\end{align*}
We obtain
\begin{align*}
\left| h_\varepsilon(x)\right|&=\frac{\prod_{i=1}^k|x^{(i)}|^{\mu_i}\ [[x]]_{\varepsilon,\lambda}^{(1+\sum_{i=1}^k(\sigma_i-1))-s} 
 }{\prod_{i=1}^k \lambda^\varepsilon_i(x)\ \|x\|_{\varepsilon,\lambda}^t},\\
\textnormal{div}_{\lambda} h_\varepsilon(x)&=
\frac{\prod_{i=1}^k|x^{(i)}|^{\mu_i}}{\|x\|_{\varepsilon,\lambda}^t [[x]]_{\varepsilon,\lambda}^s}
\bigg(c_\varepsilon(x)-t \frac{1}{\|x\|_{\varepsilon,\lambda}} \sum_{i=1}^k\frac{\lambda_i(x)}{\lambda_i^\varepsilon(x)}\sigma_i x^{(i)} \cdot\triangledown_{x^{(i)}}(\|x\|_{\varepsilon,\lambda} )\bigg)\\
&=\frac{\prod_{i=1}^k|x^{(i)}|^{\mu_i}}{\|x\|_{\varepsilon,\lambda}^t [[x]]_{\varepsilon,\lambda}^s}
\bigg(c_\varepsilon(x)-\eta_\varepsilon(x)\bigg),
\end{align*}
where $c_\varepsilon$ was defined in the proof of Theorem \ref{thm_semi} and 
$$
\eta_\varepsilon(x):=t \frac{1}{\|x\|_{\varepsilon,\lambda}} \sum_{i=1}^k\frac{\lambda_i(x)}{\lambda_i^\varepsilon(x)}\sigma_i x^{(i)} \cdot\triangledown_{x^{(i)}}(\|x\|_{\varepsilon,\lambda} ).
$$ 
We observe that 
\begin{align*}
0\leq\eta_\varepsilon(x)&=t \frac{1}{\|x\|_{\varepsilon,\lambda}} \bigg(\sum_{i=1}^k\frac{\lambda_i(x)}{\lambda_i^\varepsilon(x)}
\frac{|x^{(i)}|^2}{|x^{(i)}|^2+\varepsilon}(|x^{(i)}|^2+\varepsilon)^{\prod_{j\neq i}\sigma_j }\bigg)\|x\|_{\varepsilon,\lambda}^{1-2\prod_{j=1}^k\sigma_k}\\
&\leq t\frac{1}{\|x\|_{\varepsilon,\lambda}}\bigg( \sum_{i=1}^k
(|x^{(i)}|^2+\varepsilon)^{\prod_{j\neq i}\sigma_j }\bigg)\|x\|_{\varepsilon,\lambda}^{1-2\prod_{j=1}^k\sigma_k}=t
\end{align*}
and consequently, it follows from the proof of Theorem \ref{thm_semi} that
$$
c_\varepsilon(x)-\eta_\varepsilon(x) \geq N_1+\mu_1-s -t>0.
$$
Moreover, we have
$$
\lim_{\varepsilon\rightarrow 0}(c_\varepsilon(x)-\eta_\varepsilon(x))=Q+ \sum_{i=1}^k\sigma_i\mu_i-s-t.
$$
By our assumptions $Q>s+t-\sum_{i=1}^k\sigma_i\mu_i,$ which implies that $\textnormal{div}_{\lambda} h_\varepsilon>0$
for all sufficiently small $\varepsilon>0.$ 
Lemma \ref{lem0} applied to the function $h_\varepsilon$ leads to the inequality 
\begin{align*}
&\frac{1}{p^p}\int_\Omega(c_\varepsilon(x)-\eta_\varepsilon(x))\frac{\prod_{i=1}^k(|x^{(i)}|^{\mu_i}}{\|x\|_{\varepsilon,\lambda}^t[[x]]_{\varepsilon,\lambda}^s}|u(x)|^p\ dx\\
 \leq &\ \int_\Omega\frac{1}{(c_\varepsilon(x)-\eta_\varepsilon(x))^{(p-1)}}
\psi_\varepsilon(x)\left|\triangledown_{\lambda} u(x)\right|^p\ dx,
\end{align*}
where 
\begin{align*}
\psi_\varepsilon(x)&=\frac{\prod_{i=1}^k|x^{(i)}|^{\mu_i}}{\|x\|_{\varepsilon,\lambda}^t [[x]]_{\varepsilon,\lambda}^{s}}
\bigg(\sum_{i=1}^k\frac{\sigma_i^2|x^{(i)}|^2}{\lambda_i^\varepsilon(x)^2}\bigg)^{\frac{p}{2}}\\
&\leq \frac{\prod_{i=1}^k|x^{(i)}|^{\mu_i}}{\|x\|_{\lambda}^t [[x]]_{\lambda}^{s}}
\bigg(\sum_{i=1}^k\frac{\sigma_i^2|x^{(i)}|^2}{\lambda_i(x)^2}\bigg)^{\frac{p}{2}}=\psi(x).
\end{align*}
By taking the limit $\varepsilon$ tends to zero the statement of the theorem follows from the dominated convergence theorem.
\end{proof}

\begin{remark}
The first condition on the exponents in Theorem 
\ref{thm_dist} allows to derive the uniform estimates for $\eta_\epsilon(x)$ in the proof, while the condition
\eqref{cond2} ensures that 
$\psi$ belongs to $L^1_{loc}(\Omega).$
\end{remark}

Finally, we formulate Hardy type inequalities without weights.

\begin{theorem}
Let  $N_1>p>1.$ Then, for every $u\in\mathring{W}_\lambda^{1,p}(\Omega)$ we have  
\begin{align*}
\left(\frac{N_1-p}{p}\right)^p\int_\Omega\frac{|u(x)|^p  }{|x^{(1)}|^{p}}\ dx
&\leq\int_\Omega |\triangledown_{\lambda} u(x)|^p\ dx,\\
\left(\frac{N_1-p}{p}\right)^p\int_\Omega\frac{|u(x)|^p  }{\|x\|_{\lambda}^{p}}\ dx
&\leq\int_\Omega |\triangledown_{\lambda} u(x)|^p\ dx.
\end{align*}
\end{theorem}

\begin{proof}
It suffices to prove the first inequality.
The second inequality is an immediate consequence of the first, since the norms satisfy 
$\|x\|_{\lambda}\geq |x^{(1)}|,$ $x\in\R^N.$
We define the function 
$$
h_\varepsilon(x):=\frac{1}{(|x^{(1)}|^2+\varepsilon)^{\frac{p}{2}}}\left(x^{(1)},0,\dots,0\right)
$$
and compute 
\begin{align*}
\textnormal{div}_{\lambda} h_\varepsilon(x)&=\frac{N_1-p\frac{|x^{(1)}|^2}{|x^{(1)}|^2+\varepsilon}}{(|x^{(1)}|^2+\varepsilon)^{\frac{p}{2}}}>0,\\
|h_\varepsilon(x)|&=\frac{|x^{(1)}|}{(|x^{(1)}|^2+\varepsilon)^{\frac{p}{2}}}.
\end{align*}
Since $N_1>p$ we have $\textnormal{div}_{\lambda} h_\varepsilon>0,$ 
and Lemma \ref{lem0} applied to $h_\varepsilon$ yields the inequality
\begin{align*}
&\frac{1}{p^p}\int_\Omega \left(N_1-p \frac{|x^{(1)}|^2}{|x^{(1)}|^2+\varepsilon}\right)\frac{|u(x)|^p  }{(|x^{(1)}|^2+\varepsilon)^{\frac{p}{2}}}\ dx \\
\leq & \int_\Omega\left(N_1-p \frac{|x^{(1)}|^2}{|x^{(1)}|^2+\varepsilon}\right)^{-(p-1)} \frac{|x^{(1)}|^p}{(|x^{(1)}|^2+\varepsilon)^{\frac{p}{2}}}
|\triangledown_{\lambda} u(x)|^p\ dx.
\end{align*}
The first inequality of the theorem now follows from the dominated convergence 
theorem by taking the limit $\varepsilon$ tends to zero.
\end{proof}

\begin{appendix}
\section{ Some Remarks on the Optimality of the Constant}\label{sec_optimal}

For the particular case of Grushin type operators D'Ambrosio 
proved in \cite{DAm} that the constants in the inequalities  in Theorem \ref{thm_semi} are optimal. 
The optimality was shown similarly to the classical case 
using the explicit form of the function for which the Hardy inequality becomes an equality. 
This function does not belong to the Sobolev space $H_0^1(\Omega),$ but an approximating sequence in $H_0^1(\Omega)$
is used in the proof. Moreover, the function is strongly related to the fundamental solution  at the origin.
For more general $\Delta_\lambda$-Laplacians this function as well as the fundamental solution
are unknown, and at present we are not able to prove that our Hardy type inequalities are sharp.
\\

Using the fundamental solution at the origin
the following observations yield a simple proof for Hardy inequalities. 
We will only consider the case $p=2$ here.

Let $\lambda$ be of the form \eqref{lam}, $\Omega\subset\R^N$ be a domain,  $N\geq 3,$ and $\Phi$ be the fundamental solution at the origin of 
$-\Delta_\lambda$ on $\Omega,$ i.e.,
\begin{align*}
-\Delta_\lambda \Phi&=c\delta_0,\\
\Phi&>0,
\end{align*}
for some constant $c>0,$ where $\delta_0$ denotes the Dirac delta function. 
Moreover, let $u\in C^1_0(\Omega)$ and $v:=u\Phi^{-\frac{1}{2}}.$
Then,  the following identities follow from integration by parts and the properties of the fundamental solution 
(see \cite{AdSe} for the case of the classical Laplacian),
\begin{align}\label{fundamental}
\begin{split}
\int_\Omega |\triangledown_\lambda u|^2dx&=\frac{1}{4} \int_\Omega\frac{|\triangledown_\lambda \Phi|^2}{|\Phi|^2}u^2dx+
\frac{1}{2} \int_\Omega\triangledown_\lambda \Phi \triangledown_\lambda(v^2)dx+\int_\Omega|\triangledown_\lambda v|^2\Phi dx \\
&=\frac{1}{4} \int_\Omega\frac{|\triangledown_\lambda \Phi|^2}{|\Phi|^2}u^2dx+
\frac{1}{2} cv^2(0)+\int_\Omega|\triangledown_\lambda v|^2\Phi dx \\
&=\frac{1}{4} \int_\Omega\frac{|\triangledown_\lambda \Phi|^2}{|\Phi|^2}u^2dx
+\int_\Omega|\triangledown_\lambda v|^2\Phi dx\geq\frac{1}{4} \int_\Omega\frac{|\triangledown_\lambda \Phi|^2}{|\Phi|^2}u^2dx,
\end{split}
\end{align}
where we used that $v(0)=u(0)\Phi(0)^{-\frac{1}{2}}=0.$ 

%

The fundamental solution at the origin for the Grushin-type operator
$$
\Delta_\lambda=\Delta_x+|x|^{2\alpha}\Delta_y,\qquad \alpha\geq 0,\ z= (x,y)\in\R^{N_1}\times\R^{N_2}
$$
is of the form
$$
\Phi(x,y)=\frac{c}{[[(x,y)]]_\lambda^{Q-2}},
$$
for some constant $c\geq0$ (see \cite{DAmLu}).  
The estimate \eqref{fundamental} 
implies the weighted Hardy type inequality 
\begin{align*}
\int_\Omega |\triangledown_\lambda u(z)|^2dz&\geq \frac{1}{4} \int_\Omega\frac{|\triangledown_\lambda \Phi(z)|^2}{|\Phi(z)|^2}u(z)^2\, dz= 
\frac{(Q-2)^2}{4} \int_\Omega\frac{|x|^{2\alpha}}{[[(x,y)]]_\lambda^{2(1+\alpha)}}u(z)^2dz,
\end{align*}
which is a particular case of the inequalities in Theorem \ref{thm_semi}.
To show the optimality of the constant we consider the identity
\begin{align*}
\int_\Omega\left| \triangledown_\lambda u(z) -\varphi(z)u(z)\right|^2 dz
&=\int_\Omega\left| \triangledown_\lambda u(z)\right|^2 +|u(z)|^2 \left( |\varphi(z)|^2 + \textnormal{div}_\lambda\varphi(z) \right) dz
\end{align*}
and observe that the function 
$$
\varphi(x,y)=-\frac{Q-2}{2}\frac{|x|^{2\alpha}}{[[(x,y)]]_\lambda^{2(1+\alpha)}}\left(x,\frac{(1+\alpha)y}{|x|^\alpha}\right),
$$
which we applied in the proof of Theorem \ref{thm_semi},
satisfies 
$$
 |\varphi(x,y)|^2 + \textnormal{div}_\lambda\varphi(x,y)=-\left(\frac{Q-2}{2}\right)^2\frac{|x|^{2\alpha}}{[[(x,y)]]_\lambda^{2(1+\alpha)}}.
$$
A solution of the equation
$$
\triangledown_\lambda u(x,y) =-\frac{Q-2}{2}\frac{|x|^{2\alpha}}{[[(x,y)]]_\lambda^{2(1+\alpha)}}\left(x,\frac{(1+\alpha)y}{|x|^\alpha}\right)u(x,y)
$$
is the function
$$
u(x,y)=\frac{1}{[[(x,y)]]_\lambda^{\frac{Q-2}{2}}},
$$
which was used in \cite{DAm} to prove the optimality of the constant. 
It transforms the Hardy inequality into an equality, but does not belong to the class $\mathring W_\lambda^{1,2}(\Omega)$ if the 
domain $\Omega$ contains the origin (see \cite{DAm}, p.728).
\\

The fundamental solution for general $\Delta_\lambda$-Laplacians is unknown. 
Assuming that there exists a homogeneous distance from the origin $d_\lambda$ such that the 
fundamental solution is given by $\Phi=d_\lambda^{2-Q}$ we obtain
$$
\frac{|\triangledown_\lambda \Phi(x)|^2}{|\Phi(x)|^2}=(Q-2)^2\frac{|\triangledown_\lambda d_\lambda(x)|^2}{|d_\lambda(x)|^{2}},
$$
and \eqref{fundamental} implies the Hardy type inequality
\begin{align*}
\int_\Omega |\triangledown_\lambda u(x)|^2dx&\geq \frac{(Q-2)^2}{4} \int_\Omega\frac{|\triangledown_\lambda d_\lambda(x)|^2}{|d_\lambda(x)|^2}|u(x)|^2\, dx.
\end{align*}
Consequently, if the fundamental solution was known we could define the distance $d_\lambda:=\Phi^{\frac{1}{2-Q}}$
and compute explicit, weighted Hardy inequalities.

On the other hand,  suitable to analyze the optimality of the constants in our family of Hardy type inequalities 
is the relation
\begin{align}\label{id}
\int_\Omega\left|\frac{\varphi (x)}{\psi(x)}u(x)-\psi(x) \triangledown_\lambda u(x)\right|^2 dx =
\int_\Omega\psi(x)^2 \left| \triangledown_\lambda u(x)\right|^2 +
u(x)^2\left( \frac{|\varphi (x)|^2}{\psi(x)^2} + \textnormal{div}_\lambda \varphi(x)\right)dx,
\end{align}
which follows from integration by parts, where $\varphi:\R\rightarrow\R^N$ is a vector field and $\psi:\R\rightarrow\R$
a scalar function. 
Comparing with the first inequality in Theorem \ref{thm_semi} 
we choose
$$
\psi(x)^2 =\frac{[[x]]_\lambda^{2(1+\sum_{i=1}^k(\sigma_i-1))-s}}{\prod_{i=1}^k|x^{(i)}|^{2\sum_{j=1}^k\alpha_{ji}- \mu_i}},
$$
and observe that the function 
 $$
\varphi(x)=-\frac{Q-s+\sum_{i=1}^k\sigma_i\mu_i}{2}
\frac{\prod_{i=1}^k|x^{(i)}|^{\mu_i}}{[[x]]_\lambda^s}
\left(\frac{\sigma_1x^{(1)}}{\lambda_1(x)},\dots, \frac{\sigma_kx^{(k)}}{\lambda_k(x)}\right),
$$
which we used to prove the theorem, satisfies 
$$
\left( \frac{|\varphi (x)|^2}{\psi(x)^2} + \textnormal{div}_\lambda \varphi(x)\right)=
-\left(\frac{Q-s+\sum_{i=1}^k\sigma_i\mu_i}{2}\right)^2\frac{\prod_{i=1}^k|x^{(i)}|^{\mu_i}}{[[x]]_\lambda^s}.
$$
Consequently, the Hardy type inequality in Theorem \ref{thm_semi} is an equality if $u$ is a solution of the equation 
$$
\triangledown_\lambda u(x)=\frac{\varphi (x)}{\psi(x)^2}u(x),
$$
i.e.,
\begin{align}\label{problem}
\triangledown_{x^{(i)}}u(x)=-\frac{Q-s+\sum_{i=1}^k\sigma_i\mu_i}{2} \frac{\prod_{j\neq i}\lambda_j(x)^2}{\|x\|_\lambda^{2(1+\sum_{i=1}^k(\sigma_i-1))}}\sigma_i x^{(i)}u(x),
\qquad i=1,\dots,k.
\end{align}
Except for Grushin type operators we are unable to solve this equation, not even for the particular $\Delta_\lambda$-Laplacians in Examples \ref{ex1} to \ref{ex3}. 

\end{appendix}

\section*{Acknowledgements}
We would like to thank Prof. Enrique Zuazua for valuable discussions and remarks and
Prof. Ermanno Lanconelli and Prof. Enzo Mitidieri for some helpful comments.

\end{document}